\documentclass{elsarticle}
\usepackage{geometry}
\usepackage{ulem}
\usepackage{amssymb,amsbsy,amsmath,amsfonts,amssymb,amscd,amsthm}
\usepackage{xcolor}
\newcommand{\R}{\mathbb{R}}

\newcommand{\Z}{\mathbb{Z}}

\newcommand{\T}{\mathbb{T}}
\newcommand{\E}{\mathbb{E}}

\newcommand{\beq}{\begin{equation}}
\newcommand{\eeq}{\end{equation}}

\def\b{\beta}

\def\G{\Gamma}

\def\t{\tau}

\def\pd{\partial}

\newcommand{\cP}{{\cal P}}

\newcommand{\fDT}{D_{[t,T)}^{1-\beta}}

\newcommand{\fDz}{D_{(0,t]}^{1-\beta}}
\newcommand{\fdT}{\partial_{[t,T)}^{\beta}}
\newcommand{\fdz}{\partial_{(0,t]}^{\beta}}

\newcommand{\fIz}{I_{(0,t]}^{\beta}}

\newcommand{\diver}{{\rm div}}
\newtheorem{theorem}{Theorem}[section]
\newtheorem{lemma}[theorem]{Lemma}
\newtheorem{definition}[theorem]{Definition}
\newtheorem{proposition}[theorem]{Proposition}

\newtheorem{remark}[theorem]{Remark}

\numberwithin{equation}{section}
\makeatletter
\def\ps@pprintTitle{%
   \let\@oddhead\@empty
   \let\@evenhead\@empty
   \let\@oddfoot\@empty
   \let\@evenfoot\@oddfoot
}
\makeatother
\begin{document}
\begin{frontmatter}
\title{\Large \bf Variational time-fractional Mean Field Games}

\author{Qing Tang}  
\address{China University of Geosciences, Wuhan, China}
\ead{tangqingthomas@gmail.com} 

\author{Fabio Camilli}
\address{Universit\`{a} di Roma ``La Sapienza", Italy}
\ead{camilli@sbai.uniroma1.it}

\begin{abstract}
	We consider the variational structure of a time-fractional second order Mean Field Games (MFG) system. The MFG system consists of   time-fractional Fokker-Planck  and   Hamilton-Jacobi-Bellman equations.  In such a situation the individual agent follows a non-Markovian  dynamics given by  a subdiffusion process. Hence, the results of this paper extend the theory of variational MFG to the subdiffusive situation.  
\end{abstract}
 \begin{keyword}
 subdiffusion;  subordinator; optimal control; fractional Fokker-Planck equation;  fractional Hamilton-Jacobi-Bellman equation; variational Mean Field Games.
 \end{keyword}

\end{frontmatter}

\section{Introduction}
 The theory of Mean Field Games (MFGs for short) studies the interactions among  large number of rational and indistinguishable (symmetric) agents, each trying to minimize an objective cost function simultaneously. The dynamics is very complex, but model reduction is possible by assuming each agent's impact on the macroscopic dynamics is negligible and the number of agents $N$ tends to $\infty$. Each agent is assumed to have rational anticipations and chooses its control strategy by taking into account in its cost function the collective behaviors of other agents (congestion effect, for example) in the form of a probability distribution $m$ of the population. The Nash equilibrium for the differential game with a large number of agents  is modeled by a coupled system consisting of a backward Hamilton-Jacobi-Bellman (HJB) equation and a forward Fokker-Planck (FP) equation. The HJB equation describes the value function of each agent and the FP equation is used to describe the evolution of the probability measure
 driven by the optimal control. The theory of Mean Field Games originated from the works of Lasry and Lions \cite{Lasry2006,Lasry2007} and  independently  started by Huang, Caines and Malham\'e \cite{Huang2007}. For a general introduction of Mean Field Games we refer to \cite{Cardaliaguet2010, Gueant2010,Gomes2014,Carmona2018}.\\
\indent The theory of variational MFG is based on the dynamic formulation of the  optimal transport problem by Benamou and Brenier \cite{Benamou2000}. The essential idea is to show that MFG system can be viewed as an optimality condition for two convex problems, the first one being an optimal control Hamilton-Jacobi equations, the second one an optimal control problem for a Fokker-Planck equation. Variational  approach has already been considered in one of the first papers in mean field games \cite{Lasry2007}. In \cite{Cardaliaguet_2015}, Cardaliaguet obtained the existence and uniqueness of a weak solution for first order mean field game systems with local couplings by variational methods. From the variational approach Cardaliaguet et al. also considered the degenerate second order MFGs \cite{Cardaliaguet2015}. Our research on variational MFG is mainly motivated by the theory developed in \cite{Benamou2017} by Benamou, Carlier and Santambrogio. Recent progress in this direction include MFG systems with density constraints \cite{Meszaros2015}, entropy minimization \cite{Benamou2018} and stable solutions to MFG systems \cite{Briani2018}.\\
\indent In this work, we study the a time-fractional variational MFG and derive by optimality conditions the time-fractional MFG system   
\begin{equation}\label{MFG}
\left\{
\begin{array}{lll}
-\partial_tu+ \fDT[- \Delta u  +\frac{1}{2}|\nabla u|^2]=f(x,m),& (t,x) \in (0,T) \times \R^d\\[4pt]
\partial_t m -  [ \Delta \cdot+ \diver (\nabla u\cdot)] (\fDz m) = 0,\\
m(0,x) = m_0(x) ,\quad u(T,x) = u_T(x),
\end{array}
\right.
\end{equation}
where $\fDT$ and $\fDz$ denote the backward and  forward Riemann-Liouville fractional derivatives. To avoid complications arising from boundary conditions we work in the flat $d$-dimensional torus  $\T^d=\R^d\backslash \Z^d$.  The
term $f$  associates to a probability density $m$ a real valued function $f(x,m)$. 
\par

\indent This system has first been proposed by Camilli, De Maio in \cite{Camilli2018} to model MFGs with agents'  dynamics under subdiffusive regime.The special feature of this MFG system can be summarized as follows:
\begin{itemize}
\item[i.] On a microscopic level, the dynamics of each single agent is governed by a continuous time random walk (CTRW) dynamics such that the agent pauses for a certain waiting time before resuming motion. The waiting time is described by a random variable distributed according to a power-law function with heavy-tail. The limit process is described as a time-changed stochastic differential equation where the new time scale is given by an inverse stable subordinator. Important feature: the dynamics is non-Gaussian and non-Markovian. 
\item[ii.] On a macroscopic level, the evolution of the probability distribution of the agents is described by a time-fractional FP equation. The value function is characterized  by a time-fractional HJB equation. The time fractional derivatives have nonlocal structure.
\end{itemize}
\indent The theory of CTRW for modeling anomalous diffusion has been very topical in recent years in physics, biology and finance literature, e.g. \cite{Bouchaud1990,Metzler2000,Metzler1999,Magdziarz2007,Magdziarz2009,Weron2008,Magdziarz2014,Henry2010,Scalas2000}. The non-Markovian structure of subdiffusion makes it very instrumental in modelling long-memory effects, path dependent features or trading latency in finance. Some early applications of fractional calculus in finance can be found in \cite{Scalas2000}. In \cite{Bouchaud2004}, Bouchaud et al.  constructed a model in which liquidity providers (``market-makers") can act to create anti-persistence (or mean reversion) in price changes that would lead to a subdiffusive behaviour of the price and the model was calibrated with data. Recently, Bouchaud et al.  proposed an original application of subdiffusions to describe the dynamics of supply and demand in financial markets \cite{Benzaquen2018}.\\
\indent The applications in anomalous transport and diffusion modeling stimulated a surge of interest in the study of fractional (nonlocal in time) partial differential equations. Recently,  weak solution to time-fractional parabolic equations has been considered in \cite{Allen2016,Zacher2013,Liu2018}. In \cite{Zacher2008,Zacher2013}, Zacher obtained the weak maximum principle for time-fractional parabolic type equations. In \cite{Kolokoltsov2014}, Kolokoltsov et al.  studied a time-fractional Hamilton-Jacobi equation.  The viscosity solution theory for time-fractional PDEs was also developed in some recent works, e.g. \cite{Giga2017, Namba2018, Topp2017}. Camilli, De Maio et al. \cite{Camilli20182} introduced a Hopf-Lax formula for the solution of a fractional Hamilton-Jacobi equation. In \cite{Camilli2019}, Camilli and Goffi considered weak solutions in Sobolev space to time fractional Hamilton-Jacobi equations. In \cite{Ley2019}, Ley, Topp et al. considered the long time behavior of solutions to time fractional Hamilton-Jacobi equations.\\
\indent The paper is organized as follows. In Section \ref{sec_frac}, we review some basic facts about the fractional calculus, introduce the class of subdiffusive processes, obtain some facts about stochastic integration and It\^o's formula with regard to time-changed processes and use them to obtain the weak formulation of time-fractional Fokker-Planck (FP) equation. In Section \ref{sec_MFG}, we derive the MFG system \eqref{MFG} as the optimality condition of control problems driven by partial differential equations (HJB or FP equations). We also give verification results regarding the equivalence between solutions of MFG system and the overall optimization problems. Finally, we discuss some interesting directions of investigation for future work.\par

\section{Fractional calculus, CTRW and time-fractional FP equation}\label{sec_frac}
We start with a brief introduction to some basic results in fractional calculus. We refer to  Samko et al.  \cite{Samko1993} for a
comprehensive account of the theory. The idea of defining a derivative of fractional order ($\frac{1}{2}$ for example) dates back to Leibniz. This problem has also been considered by Riemann and Liouville among others in 19th century. The nonlocal operators under the name Riemann-Liouville derivative and integral are  the most important definitions in this subject to this day. In the first half of 20th century advances has been made by Hardy and Littlewood (\cite{Hardy1928}), and the work of Marchaud (\cite{Marchaud1927}). \\
\indent Throughout this section, we always assume that $\b\in (0,1)$.
For   $\phi:(0,T) \to \R$, the forward and backward Riemann-Liouville fractional integrals are defined by
\begin{align}
I^{\b}_{(0,t]} \phi:=\frac{1}{\G(\b)}\int_0^t \phi(\t)\frac{1}{(t-\t)^{1-\b}}d\t, \\
I^{\b}_{[t,T)} \phi:=\frac{1}{\G(\b)}\int_t^T \phi(\t)\frac{1}{(\t-t)^{1-\b}}d\t.
\end{align}
The fractional integrals are bounded linear operators over $L^p(0,T)$, $p\ge 1$; indeed, by H\"older's inequality,   if $f \in L^p(0,T)$, then
\[
\|I^\b_{(0,t]} \phi\|_{L^p} \leq \frac{T^{ \b}}{\b\G(\b)}\|\phi\|_{L^p}.
\]
The forward  Riemann-Liouville and Caputo derivatives are defined by
\begin{equation}\label{RL}
 D^{\b}_{(0,t]}  \phi:=\frac{d}{dt} \left[I^{1-\b}_{(0,t]} \phi \right]=\frac{1}{\G(1-\b)}\frac{d}{dt} \int_0^t \frac{\phi (\t)}{(t-\t)^\b}d\t,
   \end{equation}
 \begin{equation}\label{Def Caputo deriv}
 \fdz \phi:= I^{1-\b}_{(0,t]}\left[\phi'(\tau)\right]=\frac{1}{\G(1-\b)} \int_0^t \frac{\phi'(\tau)}{(t-\t)^\b}d\t,
  \end{equation}
while the   backward Riemann-Liouville and Caputo derivatives  are defined by
\begin{equation}
D^{\b}_{[t,T)} \phi:=-\frac{d}{dt} \left[I^{1-\b}_{[t,T)} \phi \right]=-\frac{1}{\G(1-\b)}\frac{d}{dt} \int_t^T\frac{ \phi (\t)}{(\t-t)^\b}d\t,   
\end{equation}
\begin{equation}
\fdT \phi:=-I^{1-\b}_{[t,T)} \left[\phi'(\tau) \right]=-\frac{1}{\G(1-\b)} \int_t^T \frac{\phi'(\tau)}{(t-\t)^\b}d\t.
\end{equation}
For $\beta \rightarrow 1$ the  Riemann-Liouville and Caputo derivatives of a smooth function $\phi$ converge to the classical derivative $\frac{d\phi}{dt}$, i.e. fractional derivatives are an extension of standard derivatives. In particular:
\begin{equation*}
D^{1-\b}_{(0,t]}\cdot 1=\frac{1}{\G(\b)}\frac{d}{dt} \int_0^t \frac{1}{(t-\t)^{1-\b}}d\t=\frac{t^{\b-1}}{\G(\b)}.
\end{equation*}
Since we can characterize the fractional integral of $\phi(t)$ as the Laplace convolution  \cite{Samko1993}:
$$
\fIz \phi=\phi(t)*\frac{t_+^{\b-1}}{\Gamma(\b)},
$$
 where $t_+=\max\{t,0\}$, the Laplace transforms of fractional integral and derivative can be obtained by direct computation and use of \eqref{RL}, e.g. Section 7.2, Chapter 2 \cite{Samko1993}:
\begin{align}
&\widehat{I_{(0,t]}^{\b} \phi}(k)=k^{-\b}\hat{\phi}(k),\\
&\widehat{\fDz \phi}(k)=k^{1-\b}\hat{\phi}(k)-I^{\b}_{(0,t]}\phi|_{t\rightarrow 0^+}.
\end{align}
The following integration by parts formula with fractional derivatives is well known, see e.g. (2.64) of Chapter 1 in \cite{Samko1993}.
\begin{equation}\label{byparts} 
 \int_0^T \int_{\T^d} \phi(t,x) D_{(0,t]}^{1-\beta}k(t,x)\,\,dxdt= \int_0^T \int_{\T^d} k(t,x) D_{[t,T)}^{1-\beta}\phi(t,x)\,\,dxdt,
\end{equation}
for $\phi, k\in C^1([0,T]\times \T^d)$.\\
\indent This motivates the definition of the following weak formulation of Riemann-Liouville derivatives (in the sense of distributions), e.g. Section 8, Chapter 2 \cite{Samko1993}:
\begin{definition}\label{byparts} Let   $u(t,x) \in L^1([0,T]\times \T^d)$, then
\begin{align}
\langle D_{(0,T]}^{1-\beta}u, \phi \rangle = \int_0^T \int_{\T^d}u(t,x) D_{[t,T)}^{1-\beta}\phi(t,x)\,\,dxdt, \label{bypart1}\\
\langle D_{[t,T)}^{1-\beta}u, \phi \rangle = \int_0^T \int_{\T^d} u(t,x) D_{(0,t]}^{1-\beta}\phi(t,x)\,\,dxdt, \label{bypart2}
\end{align}
for every $\phi(t,x) \in C^{\infty}_{c}([0,T]\times \T^d)$, where $\langle,\rangle$ denotes the duality between $C^{\infty}_{c}([0,T]\times \T^d)$ and distributions.
\end{definition}

\indent We proceed to the consider the dynamics of a single agent under subdiffusive regime. We denote by $X_t$ the position of particle at time $t$ with initial position  $x_0$, such that :
\begin{equation}\label{sdegno}
\left\{
\begin{array}{ll}
dY_t = v(D_t, Y_t)dt + \sqrt{2}\,dB_t,\\
Y_0 = x_0,\, D_0 = 0,
\end{array}
\right.
\end{equation}
and 
\begin{equation}\label{E_t}
   X_t=Y_{E_t}.
\end{equation}
where  $B_t$ is a Brownian motion in $\T^d$ and
$E_t$ is  the inverse of a $\beta$-stable subordinator $D_{t}$, i.e.
\begin{equation} \label{E_t}
E_t := \inf\{\tau>0 : D_\tau>t\},\hspace{2 mm}t\geq 0.
\end{equation}
The drift term $v\in L^{\infty}([0,T];W^{1,\infty}(\T^d))$. \\
\indent From physics perspective, the subordinated process $X_t$ can be interpreted in the following sense (\cite{Magdziarz2009}): $t$ is as an external  time scale, while the subordinator $E_t$ is a as an internal time scale which introduces trapping events in the motion. Between two jumps when the particle is not trapped, the process moves according to a standard diffusion process $Y_t$ since it holds $ D_{E_t}=t$. \par
\indent Let $\vartheta (t,\tau)$ denotes the density of process $D_t$. Then the Laplace transform of $D_t$ satisfies 
\begin{equation}\label{D_t transform}
 \E[e^{-k D_{t}}]=\int_0^{\infty}e^{-kt} \vartheta (t,\tau)dt=e^{-\tau k^\beta}.
\end{equation}
The process $E_t$  is continuous and nondecreasing, moreover  for any $t, \gamma > 0$  its  $\gamma$-moment is given by
\begin{equation*}
\mathbb{E}(E_t^{\gamma}) =\frac{t^{\beta\gamma}}{\G (\b+1)}\,.
\end{equation*}
 Note that the process $E_t$ does not have stationary and independent increments. And because of this the stochastic process $X_t$ is non-Markovian.\\
\indent The following FP equation has been widely used in the literature for describing the evolution of the law of subdiffusion processes while the dynamics of each individual agent is described by \eqref{sdegno} and \eqref{E_t}:
\begin{equation}\label{FPE}
\left\{
\begin{array}{lll}
\partial_t m -  [ \Delta \cdot+ \diver (v(t,x) \cdot)] (\fDz m) = 0,\\
m(0,x) = m_0(x).
\end{array}
\right.
\end{equation}
\indent The derivation of equation \eqref{FPE} can be found in \cite{Metzler1999,Magdziarz2014,Henry2010}.
In this paper, the notion of solution to the fractional FP equation \eqref{FPE} will be  in the sense of distributions.\\
\indent We denote by $\mathcal{P}_1(\T^d)$ the set of Borel probability measures over $\T^d$. It is endowed with the Kantorovich-Rubinstein distance (which metricizes the weak* convergence):
\begin{align*}
d(m_{1},m_{2}):=\sup_{\phi} \int_{\T^{d}} \phi \,d(m_{1}-m_{2}),
\end{align*}
where the supremum is taken over the set of Lipschitz continuous maps $\phi:\T^d \rightarrow \R$ which are Lipschitz continuous with constant $1$.\\

\begin{definition}\label{weak2}
Given $m_0\in \cP_1(\T^d)$,  $m\in L^1([0,T],\cP_1(\T^d))$ is said to be a weak solution to \eqref{FPE} with the initial condition $m( 0)=m_0 \in\cP_1(\T^d)$ if for any test function $\phi\in C^\infty_c([0,T)\times \T^d)$, we have
\[\int_{\T^d}\phi(0,x)m(0,x)dx+\int_0^T\int_{\T^d}\Big[\pd_t\phi+\fDT\big( \Delta \phi   + v(t,x)\cdot \nabla \phi\big)   \Big]m(t,x)dxdt=0.\]
\end{definition}

\begin{theorem}\label{Th_Weak_FP}
If $m_0$ is the law of $X_0$, then the law of  $X_t $ is a weak solution of the fractional FP equation  \eqref{FPE} in $C^{\beta/2}([0,T];\cP(\T^d))$. 
\end{theorem}

\indent We proceed to give the analytic proof of the mass preservation and positivity of solutions to equation \eqref{FPE}. 
\begin{lemma}\label{mass conservation}
If $m(t,x)$ is the solution to the fractional FP equation \eqref{FPE}, then $\int_{\T^d}m(t,x)dx=\int_{\T^d}m(0,x)dx$ for all $t\in [0,T]$.
\end{lemma}
\begin{proof}
Integrating on domain $\T^d$ with respect to $x$ for equation \eqref{FPE},
$$
\int_{\T^d}\partial_t m dx-  \int_{\T^d} \Delta (\fDz m)dx- \int_{\T^d}\diver (v (\fDz m))dx = 0.
$$
From integration by parts,
$$
\int_{\T^d} \Delta (\fDz m)dx=\int_{\T^d} \diver(\nabla \fDz m)dx=0,
$$
and
$$
\int_{\T^d}\diver (v(\fDz m))dx = 0.
$$
Therefore
$$
\frac{d}{dt}\int_{\T^d} mdx=\int_{\T^d}\partial_t mdx=0.
$$
By integrating from $0$ to $t$ we conclude $\int_{\T^d}m(t,x)dx=\int_{\T^d}m(0,x)dx$.
\end{proof}

\begin{lemma}\label{Positive m}
A solution $m(t,x)$ to the FP equation \eqref{FPE} preserves positivity: If $m_0(x)>0$, then $m(t,x)>0$.
\end{lemma}
\begin{proof}
The proof is based on duality argument. Take $s\in [0,T]$ and construct (backward) adjoint equation, $\psi=\psi(t,x)$:
\begin{equation}\label{adjoint equation}
\left\{
\begin{array}{ll}
\partial_t \psi+D^{1-\b}_{[t,s)}(\Delta  \psi)+D^{1-\b}_{[t,s)}(v \nabla \psi)=0,\,\,\,(t,x)\in [0,s]\times \T^d,\\
\psi(s,x)=\Psi(x).
\end{array}
\right.
\end{equation}
The terminal condition $\Psi(x)\in C^{\infty}(\T^d)$ and $\Psi(x)>0$. We note that, by taking the backward Riemann-Liouville integral equation $I^{1-\b}_{[t,s)}$ on both sides, equation \eqref{adjoint equation} can be written as a time-fractional parabolic equation with Caputo derivative:
\begin{equation}\label{adjoint equation Caputo}
\left\{
\begin{array}{ll}
\partial^{\b}_{[t,s)} \psi+\Delta  \psi+v \nabla \psi=0,\,\,\,(t,x)\in [0,s]\times \T^d,\\
\psi(s,x)=\Psi(x).
\end{array}
\right.
\end{equation}
By the  maximum principle for parabolic equation with Caputo derivatives (Theorem 3.2 of \cite{Zacher2008}), $\psi(t,x)>0$, $\forall (t,x)\in [0,s]\times \T^d$.\\
\indent We multiply \eqref{adjoint equation} by $m$ and add \eqref{FPE} multiplied by $\psi$, then integrate the resulting expression in $[0,s]\times \T^d$. We observe that:
\begin{align*}
\int_0^s\int_{\T^d}mD^{1-\b}_{[t,s)}(\Delta \psi)dxdt=\int_0^s\int_{\T^d}\psi \Delta(D^{1-\b}_{(0,t]}m)dxdt,\\
\int_0^s\int_{\T^d}mD^{1-\b}_{[t,s)}(v \nabla \psi)dxdt=-\int_0^s\int_{\T^d}\psi \diver(v D^{1-\b}_{(0,t]}m)dxdt,
\end{align*}
therefore,
\begin{equation}\label{positivity}
\int_{\T^d}m(s,x)\Psi(x)dx=\int_{\T^d}\psi(0,x)m_0(x)dx>0.
\end{equation}
Since \eqref{positivity} is satisfied for all $\Psi(x)\in C^{\infty}(\T^d)$ and $\Psi(x)>0$, by density argument we conclude that $m(s,x)>0,\,\,\,\forall s\in [0,T]$.
\end{proof}

\begin{lemma}\label{positivity}
For all $m(t,x)$ that satisfies fractional FP equation \eqref{FPE} with $m_0(x)\geq 0$, we have $\fDz m(t,x)\geq 0$ for all $t\in [0,T]$.
\end{lemma}

\begin{proof} The proof has been essentially obtained by Henry, Langlands, et al. (equation (36) (37) of \cite{Henry2010}) and Magdziarz, Gajda et al. in Theorem 1 of \cite{Magdziarz2014}. We sketch the main ideas for the readers' convenience.\\  
\indent Consider $X_t=Y_{E_t}$ as in \eqref{E_t}. The representation can be formulated as a system of stochastic differential equations:
\begin{equation}\label{stochastic system}
\begin{array}{ll}
dZ_t=dD_t,\\
dY_t=v(Z_t,Y_t)dt+\sqrt{2}dB_t,
\end{array}
\end{equation}
where $D_t$ is the $\b$-stable subordinator. Denote by $q_t(z,y)$ the joint probability density of the process $Z_t,Y_t$. We recall that  $m(t,x)$ denotes the density of process $X_t$, it can be represented as:
\begin{align}
\fDz m(t,x)=\int_0^{\infty}q_{\tau}(t,x)d\tau,\\
m(t,x)=\int_0^{\infty}I_{(0,t]}^{1-\b} q_{\tau}(t,x)d\tau.
\end{align}
Since $q_{t'}(t,x)$ is defined as a probability density function it is obvious that $\fDz m(t,x)\geq 0$.
\end{proof}

\begin{remark}
From Lemma \ref{positivity} it is obvious that the density $m(t,x)>0,\,\forall t\in (0,T]$, if $m_0(x)>0$. This gives a probabilistic interpretation of Lemma \ref{Positive m}.
\end{remark} 
\begin{remark}
In Lemma \ref{positivity}, the conclusion $\fDz m(t,x)\geq 0$ may be counterintuitive at first glance, but essential to justify the modeling with equation \eqref{FPE} and our subsequent arguments. 
Consider a subdiffusive transport system of particles with the density described by \eqref{FPE}. The momentum and kinetic energy are then described respectively $v\fDz m$ and $\frac{1}{2}|v|^2\fDz m$.
If $\fDz m(t,x)\geq 0$ is violated then negative kinetic energy may appear, which is not acceptable from physics point of view.\\
\indent Consider the case of fractional Fokker-Planck equation with time independent drift (v=v(x)). We now show $\fDz m(t,x)\geq 0$ can be obtained in a simpler way and is included in Lemma \ref{positivity} as a special case.
Denote by $\rho(t,x)$ the solution to the classical Fokker-Planck equation with the same initial condition
\begin{equation}
\left\{
\begin{array}{lll}
\partial_t \rho(t,x) -  \Delta \rho(t,x)+\diver(v(x)\rho(t,x)) = 0,\\
\rho(0,x) = m_0(x).
\end{array}
\right.
\end{equation}
There exists the well known relationship between the density functions (Metzler, Klafter, (123) (124) of \cite{Metzler2000}):
$$
\hat{m}(k,x)=\int_0^{\infty}\rho(\tau,x)k^{\beta-1}e^{-\tau k^{\beta}}d\tau,
$$
where $\hat{m}(k,x)$ denotes the Laplace transform of $m(t,x)$.
It can be verified using Laplace transform formula of Riemann-Liouville derivative, that:
$$
\widehat{D^{1-\b}_{(0,t]}m}(k,x)=k^{1-\beta}\int_0^{\infty}\rho(\tau,x)k^{\beta-1}e^{-\tau k^{\beta}}d\tau=\int_0^{\infty}\rho(\tau,x)e^{-\tau k^{\beta}}d\tau.
$$
By using Laplace transform relation \eqref{D_t transform} we obtain
$$
D^{1-\b}_{(0,t]}m(t,x)=\int_0^{\infty}\rho(\tau,x)\vartheta (t,\tau)d\tau\geq 0.
$$
\indent In particular, in the case of drift $v(x)$, the system \eqref{stochastic system} becomes uncoupled :
\begin{equation}
\begin{array}{ll}
dZ_t=dD_t,\\
dY_t=v(Y_t)dt+\sqrt{2}dB_t.
\end{array}
\end{equation}
In this scenario, we have for the joint density $q_t(z,y)=\rho(\tau,x)\vartheta (t,\tau)$. 
\end{remark}




 \section{Time-fractional Mean Field Games}\label{sec_MFG}
In this section we introduce the MFG system and the corresponding variational interpretation. The following conditions are supposed to hold throughout the rest of the paper.
\begin{itemize}
\item[(H1)] The coupling $f:\T^d \times[0,+\infty)\to \R$ is continuous in both variables, increasing with respect to the second variable $m$. Moreover the following normalization condition holds:
$$
f(x,0)=0,\,\,\,\forall x\in \T^d.
$$
\item[(H2)] $ f(x,m)$ is increasing with respect to $m$, i.e., $\forall m_1,m_2\in C([0,T];\mathcal{P}_1)$, $\forall t\in [0,T]$,
$$
\int_{\T^d}( f(x,m_1)- f(x,m_2))(m_1-m_2)\,dx\geq 0.
$$
\item[(H3)] $u_T(x):\T^d\to \R$ is of class $C^2$, while $m_0(x): \T^d\to \R$ is a $C^1$ positive density, i.e. $m_0(x)>0$ and $\int_{T^d}m_0(x)dx=1$.
\end{itemize}
\indent Let us set:
\begin{align*}
F(x,m)=\left \{
\begin{array}{ll}
\int_0^m f(x,\t)d\t, \,\,\text{if} \,\,m\geq 0,\\
+\infty \,\, \text{otherwise}.
\end{array}
\right.
\end{align*}
From conditions (H1) and (H2) it follows that $F(x,m)$ is convex with respect to $m$.\\

We now show the duality between two optimal control problems constrained, respectively, by a fractional HJB equation and a fractional FP equation. 
We start  introducing  a control problem for HJ equation. Denote by $\mathcal{K}_0$ the set of maps $u \in C^2([0,T]\times \T^d)$ such that $u(T,x)=u_T(x)$ and define, on $\mathcal{K}_0$, the functional\\
\begin{equation}\label{infHJ}
\mathcal{A}(u)=\int_0^T \int_{\T^d}  F^{*}(x,-\partial_tu+ \fDT[- \Delta u  +H(x,\nabla u) ])\,dxdt-\int_{\T^d}m_0 u(0,x)dx,
\end{equation}
the Hamiltonian is defined as $H(x,p)=\frac{1}{2}|p|^2$. Here $F^{*}$ denotes the Legendre-Fenchel transform of $F(x,m)$ such that $\forall \alpha$:
$$
F^{*}(x,\alpha)=\sup_{\alpha} \left\{\alpha m- F(x,m)\right\}.
$$
Next we formulate an optimal control problem for the FP equation. We can linearize the constraint \eqref{FPE} by introducing the variable
$$
w(t,x)=-v(t,x)\fDz m(t,x).
$$
Then the FP equation can be written as:
\begin{equation}\label{continuity eqn}
\left\{
\begin{array}{lll}
\partial_t m -   \Delta (\fDz m) + \diver (w) = 0,\,\,\text{in}\,\,(0,T)\times \T^d\\
m(0,x) = m_0(x),
\end{array}
\right.
\end{equation}
where the solution is in the sense of distributions. The Legendre-Fenchel transform of the Hamiltonian can be written as:
$$
H^*\left(x,-\frac{w(t,x)}{\fDz m(t,x)}\right)=\left\{
\begin{array}{lll}
\frac{1}{2}\left|\frac{w(t,x)}{\fDz m(t,x)}\right|^2\,\,&\text{if}\,\fDz m(t,x)>0,\\[3pt]
0 &\text{if}\,(\fDz m,w)=(0,0),\\
+\infty&\text{otherwise.}
\end{array}
\right.
$$
Let us denote $\mathcal{K}_1$ the set of pairs $(m,w)\in L^1((0,T)\times \T^d)\times L^1((0,T)\times \T^d,\,\R^d)$ such that $m(t,x)>0$, $\int_{\T^d}m(t,x)dx=1$ for a.e. $t\in (0,T)$. \\
\indent On the set $\mathcal{K}_1$, define the following functional
\begin{align*}
\mathcal{B}(m,w)=&\int_0^T\int_{\T^d}\fDz m(t,x)H^*\left(x,-\frac{w(t,x)}{\fDz m(t,x)}\right)+ F(x,m(t,x))dxdt\\
&+\int_{\T^d}u_T(x)m(T,x)dx.
\end{align*}
Since $H^*$ and $F$ are bounded from below and $\fDz m\geq 0$ a.e. (by Lemma \ref{positivity}), the first integral in $\mathcal{B}(m,w)$ is well defined in $\R \cup \{+\infty\}$.\\
\indent We proceed to our main duality result.
\begin{proposition}\label{duality}
We have
\begin{equation}\label{id_duality}
\inf_{u\in \mathcal{K}_0}\mathcal{A}(u)=-\min_{(m,w)\in \mathcal{K}_1}\mathcal{B}(m,w).
\end{equation}
\end{proposition}
\begin{proof}
Let $E_0=C^2([0,T]\times \T^d)$ and $E_1=C([0,T]\times \T^d,\, \R)\times C([0,T]\times \T^d,\,\R^d)$. Define on $E_0$ the functional 
$$
\mathcal{F}(u)=-\int_{\T^d} m_0(x)u(0,x)dx+\chi_S(u),
$$
where $\chi_S$ is the characteristic function of the set $S=\{u\in E_0,\,\, u(T,\cdot)=u_T\}$, i.e., $\chi_S(u)=0$ if $u\in S$ and $+\infty$ otherwise. For $(a,b)\in E_1$, we define:
$$
\mathcal{G}(a,b)=\int_0^T\int_{\T^d}F^*(x,-a(t,x))+\fDT H(x,b(t,x))dxdt.
$$
The functional $\mathcal{F}$ is convex and lower semi-continuous on $E_0$. Since by using Definition \ref{byparts},
$$
\int_0^T\int_{\T^d} \fDT H(x,b(t,x))dxdt=\int_0^T\int_{\T^d}  H(x,b(t,x))\fDz \cdot 1dxdt,
$$
and $\fDz \cdot 1=\frac{t^{\b-1}}{\Gamma(\b)}>0$,
thus $\int_0^T\int_{\T^d} \fDT H(x,b(t,x))dxdt$ is convex and continuous on $E_1$, therefore $\mathcal{G}$ is convex and continuous on $E_1$. Let $\Lambda: E_0\rightarrow E_1$ be the bounded linear operator defined by:
$$
\Lambda(u)=(\partial_t u+\fDT(\Delta u),\,\,\nabla u).
$$
Then we obtain 
$$
\inf_{u\in \mathcal{K}_0}\mathcal{A}(u)=\inf_{u\in \mathcal{K}_0}\{\mathcal{F}(u)+\mathcal{G}(\Lambda(u))\}.
$$
It follows by Fenchel-Rockafellar duality theorem that 
$$
\inf_{u\in \mathcal{K}_0}\{\mathcal{F}(u)+\mathcal{G}(\Lambda(u))\}=\max_{(m,w)\in E_1^{'}} \{ -\mathcal{F}^*(\Lambda^*(m,w))-\mathcal{G}^*(-(m,w)) \}
$$
where $E_1^{'}$ is the dual space of $E_1$, i.e., the set of vector valued Radon measures $(m,w)$ over $[0,T]\times \T^d$ with values in $\R \times \R^{d}$, $E_0^{'}$ is the dual space of $E_0$, $\Lambda^*: E_1^{'}\rightarrow E_0^{'}$ is the dual operator of $\Lambda$ and $\mathcal{F}^{*}$ and $\mathcal{G}^{*}$ are the convex conjugate of $\mathcal{F}$ and $\mathcal{G}$ respectively. From integration by parts 
$$
\int_0^T\int_{\T^d} m\fDT(\Delta u)dxdt=\int_0^T\int_{\T^d} u\Delta  (\fDz m)dxdt,
$$
we obtain:
\begin{align*}
&\mathcal{F}^*(\Lambda^*(m,w))\\
=&\sup_{u\in E_0}\{\langle \Lambda^*(m,w),u\rangle-\mathcal{F}(u)\} \\
=&\sup_{u\in E_0}\{\langle (m,w),\Lambda u\rangle-\mathcal{F}(u)\}\\
=&\sup_{u\in E_0}\{\int_0^T\int_{\T^d} m(\partial_tu+\fDT(\Delta u))dxdt+\int_0^T\int_{\T^d}w\nabla udxdt-\mathcal{F}(u)\}\\
=&\sup_{u\in E_0}\{\int_0^T\int_{\T^d} u(-\partial_tm+\Delta(\fDz m)-\diver w)dxdt+\int_{\T^d}(m_0(x)-m(0,x))u(0,x)dx\\
&-\chi_S(u)+\int_{\T^d}u(T,x)m(T,x)dx\}.
\end{align*}
Therefore
\begin{align*}
\mathcal{F}^*(\Lambda^*(m,w))=\left\{
\begin{array}{lll}
\int_{\T^d} u_T(x)m(T,x)dx\quad &\text{if $m,w$  is a solution of \eqref{continuity eqn},}\\[4pt]
+\infty& \text{otherwise}.
\end{array}
\right.
\end{align*}
Moreover 
\begin{align*}
&\mathcal{G}^{*}(m,w)\\
=&\sup_{a\in \R,b\in \R^d} \left\{\int_0^T\int_{\T^d}(am+\langle b,w\rangle)dxdt-\mathcal{G}(a,b)\right\}\\
=&\sup_{a\in \R,b\in \R^d} \int_0^T\int_{\T^d}(am+\langle b,w\rangle)dxdt-\int_0^T\int_{\T^d}F^*(x,-a+\fDT H(x,b)dxdt\\
=&\sup_{a\in \R,b\in \R^d} \int_0^T\int_{\T^d}(m\fDT H(x,b)-am+\langle b,w\rangle-F^*(x,a))dxdt\\
=&\sup_{a\in \R,b\in \R^d}\int_0^T\int_{\T^d}(H(x,b)(\fDz m)-am+\langle b,w\rangle-F^*(x,a))dxdt.
\end{align*}
Hence, for  $\fDz (-m)>0$,
\begin{align*}
&\mathcal{G}^{*}(m,w)\\
=&\sup_{(m,w)\in E_1^{'}} \int_0^T\int_{\T^d}\fDz(-m)\left(\langle b,-\frac{w}{\fDz m}\rangle-H(x,b)\right)\\
&+(-m)a-F^*(x,a)dxdt\\
=& \int_0^T\int_{\T^d} \fDz(-m)H^*\left(x,-\frac{w}{\fDz m}\right)+ F(x,-m)dxdt.
\end{align*}
We can denote:
$$
 \mathcal{G}^{*}(m,w)=\int_0^T\int_{\T^d} \mathcal{K}^{*}(m,w)dxdt,
$$
such that 
\begin{align*}
\mathcal{K}^{*}(m,w)=
\left\{
\begin{array}{lll}
\fDz(-m)H^*\left(x,-\frac{w}{\fDz m}\right)+ F(x,-m)\,\,&\text{if}\,  \fDz m<0, \\[4pt]
0 &\text{if}\,  \fDz m=0,\, w=0,\\
+\infty& \text{otherwise}.
\end{array}
\right.
\end{align*}
Therefore
\begin{align*}
&\max_{(m,w)\in E_1^{'}} \{ -\mathcal{F}^*(\Lambda^*(m,w))-\mathcal{G}^*(-(m,w)) \}\\
=&\max_{(m,w)\in E_1^{'}} \{ \int_0^T\int_{\T^d} -\fDz(m)H^*(x,-\frac{w}{\fDz m})- F(x,m)dxdt\\
&-\int_{\T^d} u_T(x)m(T,x)dx\}\\
=&-\min_{(m,w)\in E_1^{'}} \{ \int_0^T\int_{\T^d} \fDz(m)H^*(x,-\frac{w}{\fDz m})+ F(x,m)dxdt\\
&+\int_{\T^d} u_T(x)m(T,x)dx\} \\
=&-\min_{(m,w)\in E_1^{'}}\mathcal{B}(m,w).
\end{align*}
The minimum is taken over the $L^1$ maps $(m,w)$ such that $m(t,x)\geq 0$ and $\fDz m(t,x)\geq 0$ a.e. and 
\eqref{continuity eqn} holds in the sense of distributions.
%
Since $\int_{\T^d}m_0dx=1$, by Lemma \ref{mass conservation} on mass conservation we have $\int_{\T^d}m(t,x)dx=1$ for any $t\in [0,T]$. Hence the pair $(m,w)$ belongs to the set $\mathcal{K}_1$ and we have proved the duality of the optimal control problems.\\
\indent We next consider a minimizing sequence $(m_n,w_n)$ such that $\mathcal{B}(m_n,w_n)\leq C$, then 
$$
\int_0^T\int_{\T^d}\frac{w_n^2}{\fDz m_n}dxdt\leq C.
$$
It has been shown in \cite{Camilli2018} that the sequence $(m_n)$ is uniformly bounded in $C^{\frac{\beta}{2}}([0,T];\mathcal{P}(\T^d))$. In particular, by H\"{o}lder's inequality, on the set $\fDz m>0$, we have
\begin{align*}
\int_0^T\int_{\T^d}|w_n|dxdt&\leq \left(\int_0^T\int_{\T^d}\frac{w_n^2}{\fDz m_n}dxdt\right)^{\frac{1}{2}}\left( \int_0^T\int_{\T^d}\fDz m_ndxdt\right)^{\frac{1}{2}}\\
&\leq C\left(\int_0^T\fDz (\int_{\T^d}m_n(t,x)dx)dt\right)^{\frac{1}{2}}\\
&= C\left(\frac{T^{\beta}}{\beta \Gamma(\beta)}\right)^{\frac{1}{2}}.
\end{align*}
Therefore, up to a subsequence, $w_n\rightarrow w$ in $\mathcal{M}((0,T)\times \T^d,\R^d)$ and $(m_n)$ converges in $C([0,T];\mathcal{P}(\T^d))$. The limit $(m,w)$ is then the minimizer of $\mathcal{B}(m,w)$.
\end{proof}
\begin{remark}
\indent In general, the existence of a minimizer for the problem $\mathcal{A}(u)$ can be obtained as a solution to the fractional HJ equation in \eqref{MFG}. 
By reversing time (change of variable $t$ by $T-t$) the backward HJB equation is equivalent to 
\begin{equation}\label{FFHJ}
\left\{
\begin{array}{lll}
\partial_tu+ \fDz[- \Delta u  +\frac{1}{2}|\nabla u|^2]=f(x,m),& (t,x) \in (0,T) \times \R^d\\[4pt]
u(0,x) = u_T(x).
\end{array}
\right.
\end{equation}
By taking the fractional integral on both sides of \eqref{FFHJ}, use \eqref{Def Caputo deriv} such that
\begin{align*}
&I^{1-\b}_{(0,t]}(\partial_tu)=\fdz u,\\
&I^{1-\b}_{(0,t]}(\fDz[- \Delta u  +\frac{1}{2}|\nabla u|^2])=- \Delta u  +\frac{1}{2}|\nabla u|^2,
\end{align*}
we obtain 
\begin{equation}\label{HJ_CG}
\left\{
\begin{array}{lll}
\fdz u- \Delta u  +\frac{1}{2}|\nabla u|^2=I^{1-\b}_{(0,t]} f(x,m),& (t,x) \in (0,T) \times \R^d\\[4pt]
u(0,x) = u_T(x).
\end{array}
\right.
\end{equation}
Existence  and uniqueness of a classical solution to \eqref{HJ_CG} (and therefore to \eqref{FFHJ}) has been recently studied in \cite{Camilli2019}. If there exists a classical solution to the backward HJ equation in \eqref{MFG}, then the vector field $\nabla u$ is regular and also the FP equation admits a unique classical solution.
\end{remark}
In the next theorem, we show the connection between the optimal control problems for the fractional FP and HJB equations  and the MFG system \eqref{MFG}.
\begin{theorem} \label{th:OP_MFG}
Assume that $(\bar{m},\bar{u})$ is of class $C^1([0,T]\times \mathcal{P}_1(\T^d))\times C^1([0,T]\times \T^d)$, with $\bar{m}(0,x)=m_0$ and $\bar{u}(T,x)=u_T(x)$.
Then the following statements are equivalent:\\
(i) $(\bar{m},\bar{u})$ is a solution to the fractional MFG system \eqref{MFG}.\\
(ii) The solution $\bar{u}(x,t)$ is optimal for $\inf_u \mathcal{A}$.\\
(iii) The control $\bar{v}=-\nabla \bar{u}(t,x)$ and $\bar{m}$ are optimal for $\min_{(m,w)}\mathcal{B}$ where $\bar{w}=-\bar{v}\fDz \bar{m}$, $\bar{m}$ is the solution of FP equation \eqref{FPE}.
\end{theorem}
\begin{proof} The proof is by verification arguments.\\
``i $\Rightarrow$  ii": assume that $(\bar{m},\bar{u})$ is solution to MFG system \eqref{MFG}. Denote by $\bar{\alpha}$ and $\alpha$ respectively
\begin{align*}
\bar{\alpha}=-\partial_t \bar{u}-\fDT  \Delta \bar{u}+\frac{1}{2}\fDT (|\nabla \bar{u}|^2),\\
\alpha=-\partial_t u-\fDT  \Delta u+\frac{1}{2}\fDT (|\nabla u|^2)\\
\bar{u}(T,x)=u(T,x)=u_T(x,m).
\end{align*}
Thus we can write 
\begin{align*}
\mathcal{A}(u)=J^{HJ}(\alpha)=\int_0^T \int_{\T^d}F^{*}(x,\alpha)\,dxdt-\int_{\T^d}u(0,x)m_0(x)dx
\end{align*}
By definition of Legendre transform $F^*(x,\alpha)$ we know it is convex with regard to $\alpha$, hence 
\begin{align*}
J^{HJ}(\alpha)\geq& J^{HJ}(\bar{\alpha})+\int_0^T \int_{\T^d}  \partial_{\alpha} F^*(\alpha-\bar{\alpha})\,dxdt\\
&-\int_{\T^d} (u(0,x)-\bar{u}(0,x))dm_{0}(x)\\
=&J^{HJ}(\bar{\alpha})+\int_0^T \int_{\T^d} \partial_{\alpha} F^*(-\partial_t (u-\bar{u})-\fDT  \Delta (u-\bar{u})\\
&+\frac{1}{2}\fDT (|\nabla u|^2-|\nabla \bar{u}|^2))\,dxdt-\int_{\T^d} (u(0,x)-\bar{u}(0,x))m_{0}(x)dx.
\end{align*}
Since, by definition, $\partial_\alpha F^*=\bar{m}$ and  $|\nabla u|^2-|\nabla \bar{u}|^2\geq 2\nabla \bar{u}(u-\bar{u})$, we obtain
\begin{align*}
&\int_0^T \int_{\T^d}\frac{1}{2}\fDT (|\nabla u|^2-|\nabla \bar{u}|^2))\,dxdt\\
=&\int_0^T \int_{\T^d}\frac{1}{2} (|\nabla u|^2-|\nabla \bar{u}|^2))\frac{t^{\b-1}}{\Gamma(\b)}\,dxdt\\
\geq &\int_0^T \int_{\T^d} \nabla \bar{u}(u-\bar{u})\frac{t^{\b-1}}{\Gamma(\b)}\,dxdt\\
= &\int_0^T \int_{\T^d}\fDT \nabla \bar{u}(u-\bar{u})\,dxdt.
\end{align*}
Thus we have
\begin{align*}
&J^{HJ}(\alpha)\\
\geq&  J^{HJ}(\bar{\alpha})+\int_0^T \int_{\T^d} \bar{m}(-\partial_t (u-\bar{u})-\fDT  \Delta (u-\bar{u})+\fDT \nabla \bar{u}(u-\bar{u}))\,dxdt\\
&-\int_{\T^d} (u(0,x)-\bar{u}(0,x))m_{0}(x)dx\\
=& J^{HJ}(\bar{\alpha})+\int_0^T \int_{\T^d} (u-\bar{u})(\partial_t \bar{m} -  [\Delta \cdot+ \diver (\nabla \bar{u}\cdot)] (\fDz \bar{m}))\,dxdt\\
&-\int_{\T^d} (u(T,x)-\bar{u}(T,x))m(T,x)dx+\int_{\T^d}(u(0,x)-\bar{u}(0,x))(\bar{m}(0,x)-m_0(x))dx.
\end{align*}
By definition $(\bar{m},\bar{u})$ is solution to MFG system \eqref{MFG}, therefore the equation
$$
\partial_t \bar{m} -  [\Delta \cdot+ \diver (\nabla \bar{u}\cdot)] (\fDz \bar{m})=0,\,\,\,\bar{m}(0,x)=m_0,
$$
is satisfied in the sense of distributions. Hence we obtain $J^{HJ}(\alpha)\geq J^{HJ}(\bar{\alpha})$. From the uniqueness of solution to the fractional HJ equation
 $$
 -\partial_t \bar{u}-\fDT  \Delta \bar{u}+\frac{1}{2}\fDT (|D\bar{u}|^2)=\bar{\alpha},\,\,\,\bar{u}(T,x)=u_T(x),
 $$
  we conclude:
$$
\mathcal{A}(u)\geq \mathcal{A}(\bar{u}).
$$
``ii$\Rightarrow$ i": Define $\bar{m}(t,x)=\partial_{\alpha} F^*(x,\bar{\alpha}(t,x))$, then $\bar{\alpha}(t,x)= f(x,\bar{m}(t,x))$. Take a smooth function $\delta \alpha \in C^1([0,T],\T^d)$, denote by $u_h$ the solution to the equation
\begin{equation}\label{uh}
\left\{
\begin{array}{lll}
-\partial_tu+ \fDT[- \Delta u  +\frac{1}{2}|\nabla u|^2 ]=\bar{\alpha}+h\delta \alpha,& (t,x) \in (0,T) \times \T^d\\[4pt]
 u(T,x) = u_T(x).
\end{array}
\right.\end{equation}
We recall that $\bar{u}$ is the solution to the system
\begin{equation}
\left\{
\begin{array}{lll}
-\partial_tu+ \fDT[- \Delta u  +\frac{1}{2}|\nabla u|^2 ]=\bar{\alpha},& (t,x) \in (0,T) \times \T^d\\[4pt]
 u(T,x) = u_T(x).
\end{array}
\right.\end{equation}
Thus $(u_h-\bar{u})/h$ converges to a smooth fucntion $\phi$ which is the solution to the linearized system:
\begin{equation*}
\left\{
\begin{array}{lll}
-\partial_t\phi+ \fDT[- \Delta \phi  +\nabla \bar{u}\cdot D\phi]=\delta \alpha(t,x),& (t,x) \in (0,T) \times \T^d\\[4pt]
 \phi(T,x) = 0.
\end{array}
\right.\end{equation*}
\indent By calculating the first variation and applying mean value theorem we can obtain 
\begin{align*}
&\frac{\delta J^{HJ}(\bar{\alpha})}{\delta \alpha}\\
=&\lim_{h\rightarrow 0}\frac{J^{HJ}(\bar{\alpha}+h\delta \alpha)-J^{HJ}(\bar{\alpha})}{h}\\
=&\int_0^T \int_{\T^d} \lim_{h\rightarrow 0} \frac{F^*(x,\bar{\alpha}+h\delta \alpha)-F^*(x,\bar{\alpha})}{h}\,dxdt- \int_{\T^d} \phi(0,x)m_0(x)dx\\
=&\int_0^T \int_{\T^d} \bar{m}\cdot \delta \alpha \,dxdt-  \int_{\T^d} \phi(0,x)m_0(x)dx\\
=&\int_0^T \int_{\T^d} \bar{m}(-\partial_t\phi+ \fDT[- \Delta w  +\nabla \bar{u}\cdot \nabla \phi]) \,dxdt-  \int_{\T^d} \phi(0,x)m_0(x)dx\\
=&\int_0^T \int_{\T^d}\phi(\partial_t \bar{m} -  [\Delta \cdot+ \diver (\nabla u\cdot)] (\fDz \bar{m}))\,dxdt-\int_{\T^d} \phi(0,x)(m_0(x)-\bar{m}(0,x))dx.
\end{align*}
Since by density argument this holds for any test function $\phi \in C^3$ with  $\phi(T,x) = 0$, thus we obtain $\bar{m}$ is a weak solution to the FP equation in MFG system \eqref{MFG} with $\bar{m}(0,x)=m_0(x)$.\\
\indent Moreover, the equivalence of (ii) and (iii) has been shown in Theorem \ref{duality}.
\end{proof}

\section{Discussions}
In this paper we have focused on studying the variational structure of time-fractional MFG system.  An important consequence  of the variational interpretation of the Mean Field Games is the possibility to show the existence of a solution   to system \eqref{MFG} under weak regularity assumptions.  With smoothing coupling term  more regular solutions may be constructed using results from recent work on time-fractional Hamilton-Jacobi equations \cite{Camilli2019}. One interesting direction will be to study the long time behavior of such time-fractional MFG systems.\\
\indent In \cite{Cesaroni2017} and \cite{Cirant2019}, the authors studied mean field games with jump diffusions. The MFG systems in these models involve fractional Laplacians. In recent works in physics, it has been shown that the existence of phenomena displaying a very interesting a L\'evy walk dynamics with both sub and super diffusion at the same time, due to cumulative inertia and long-range interactions. For modeling  jump diffusion with heavy tailed random waiting time between jumps,  in \cite{Weron2008} Weron, Magdziarz et al. proposed a the one dimensional space-time fractional Fokker-Planck equation
$$
\partial_t m=[-\frac{\partial}{\partial x}(v\cdot)+\nabla^{\mu}\cdot]\fDz m,
$$
where $\nabla^{\mu}$ denotes the Riesz fractional derivative. In our future work we will use such nonlocal Fokker-Planck equation and convex duality methods in this paper to consider space-time nonlocal MFG systems. This can be used in modeling of strategic interactions between large number of ``small'' agents, while the dynamics of each agent involves complex nonlocal in space and time structures. The study of   system  \eqref{MFG} can be considered as a first step  in the direction of extending the MFG theory to more complex, and realistic, nonmarkovian dynamics.

\section*{Acknowledgement}
 The first author would like to thank the hospitality of Universit\`{a} di Roma ``La Sapienza" during preparation of this work.\\
\indent Qing Tang is partially supported by China National Science Foundation grant No. 11701534.\\
\indent Both authors thank Mirko D'Ovidio (Universit\`{a} di Roma ``La Sapienza")  for helpful discussions about fractional calculus and subordinated processes.

\renewcommand*{\bibfont}{\footnotesize}

\end{document}